\documentclass[english]{amsart}

\usepackage[utf8]{inputenc}
\usepackage{amsthm}
\usepackage{amssymb}
\usepackage{mathtools}
\usepackage{color}
\usepackage{babel}

\newcommand{\Z}{\mathbb{Z}}

\newcommand{\D}{\mathbb{D}}
\newcommand{\cG}{\mathcal{G}}
\newcommand{\Ret}{\mathrm{Ret}}
\newcommand{\Sym}{\mathbb{S}}

\newcommand{\id}{\mathrm{id}}

\newcommand{\Aut}{\mathrm{Aut}}

\newcommand{\Soc}{\mathrm{Soc}}
\newcommand{\Ker}{\mathrm{Ker}}

\makeatletter
\numberwithin{table}{section}
\numberwithin{equation}{section}
\numberwithin{figure}{section}
\theoremstyle{plain}
\newtheorem{thm}{Theorem}[section]
\theoremstyle{plain}

\theoremstyle{plain}

\newtheorem*{question*}{Question}
\newtheorem*{conjecture*}{Conjecture}

\newtheorem{example}[thm]{Example}
\theoremstyle{remark}
\newtheorem{rem}[thm]{Remark}

\theoremstyle{plain}
\newtheorem{exa}[thm]{Example}
\newtheorem{convention}[thm]{Convention}

\makeatother

\begin{document}

\title[Multipermutation solutions]{A characterization of finite multipermutation solutions of the Yang--Baxter equation}

\author{D. Bachiller}
\author{F. Ced\'o}
\author{L. Vendramin}

\thanks{The two first-named authors are partially supported by the grant
MINECO MTM2014-53644-P.
The third-named author is partially supported by PICT-2014-1376,
MATH-AmSud 17MATH-01, ICTP,
ERC advanced grant 320974
and the Alexander von Humboldt Foundation.}

\begin{abstract}
We prove that a finite non-degenerate involutive set-theoretic
solution $(X,r)$ of the Yang--Baxter equation is a multipermutation
solution if and only if its structure group $G(X,r)$ admits a
left ordering or equivalently it is poly-$\Z$.
\end{abstract}

\maketitle

\noindent{\em 2010 MSC: Primary 16T25, 20F16, 20F60.}

\noindent\keywords{\em Keywords: Yang-Baxter equation, set-theoretic solution,
brace, ordered groups, poly-(infinite cyclic) group}

\section*{Introduction}

According to Drinfeld~\cite{MR1183474}, a set-theoretic
solution of the Yang--Baxter equation is a pair $(X,r)$, where $X$
is a set and $r\colon X\times X\to X\times X$ is a bijective map
such that
\[
(r\times\id_X)(\id_X\times r)(r\times\id_X)= (\id_X\times
r)(r\times\id_X)(\id_X\times r).
\]

The seminal papers \cite{MR1722951} and~\cite{MR1637256} initiated
the study of non-degenerate involutive set-theoretic solutions of
the Yang--Baxter equation. Etingof, Schedler and Soloviev introduced
the structure group $G(X,r)$ of a solution $(X,r)$ as the group
presented with set of generators  $X$ and with relations $xy=uv$
whenever $r(x,y)=(u,v)$. This group turned out to be very important
to understand set-theoretic solutions. As proved by Gateva-Ivanova
and Van den Bergh, the structure group $G(X,r)$ of a finite
non-degenerate involutive set-theoretic solution of the Yang--Baxter
equation is a Bieberbach group, i.e. a finitely generated
torsion-free abelian-by-finite group.

In \cite{MR1722951} multipermutation solutions were introduced. This
is an important notion that was intensively
studied~\cite{BCJO_family,MR3574204,MR2652212,MR3177933,MR2095675,
MR2885602,MR3437282}. In~\cite[Proposition 4.2]{MR2189580} Jespers
and Okni\'nski proved that the structure group of a finite
multipermutation solution is poly-$\Z$. The main result of this
paper is to prove the converse: a finite solution $(X,r)$ such that
$G(X,r)$ is poly-$\Z$ is a multipermutation solution.

To prove our result we use the language of braces introduced by Rump
in~\cite{MR2278047}.  Braces are algebraic structures that
generalize radical rings. This fact allows us to use tools and
techniques from ring theory to study set-theoretic solutions of the
Yang--Baxter equation.

In~\cite[Theorem 23]{MR639438} Farkas proved that a Bieberbach group
is poly-$\Z$ if and only if it admits a left ordering. Since
Chouraqui proved in~\cite[Theorem 1]{MR2764830} that the structure
group of a finite non-degenerate involutive set-theoretic
solution of the Yang--Baxter equation is a Garside group, our
result in particular yields an infinite family of Garside groups
that are not left orderable.

\section{Preliminaries}

A \emph{set-theoretic solution} of the Yang--Baxter equation is a
pair $(X,r)$, where $X$ is a set and $r\colon X\times X\to X\times
X$ is a bijective map such that
\[
(r\times\id_X)(\id_X\times r)(r\times\id_X)= (\id_X\times
r)(r\times\id_X)(\id_X\times r).
\]

A solution $(X,r)$ is said to be \emph{involutive} if
$r^2=\id_{X^2}$ and it is said to be \emph{non-degenerate} if
\[
r(x,y)=(\sigma_x(y),\gamma_y(x)),
\]
where $\sigma_x$ and $\gamma_x$ are permutations of $X$ for all
$x\in X$. The \emph{structure group} $G(X,r)$ of a
non-degenerate solution $(X,r)$ is defined as the group
presented with set of generators $X$ and with relations $xy=uv$
whenever $r(x,y)=(u,v)$. In~\cite[Theorem 1]{MR2764830} Chouraqui
proved that the structure group of a non-degenerate involutive
set-theoretic solution of the Yang--Baxter equation is a Garside
group. A simpler proof of this result was recently given by Dehornoy
in~\cite{MR3374524}.

\begin{example}
\label{exa:rack} Let $X$ be a conjugacy class of a finite group $G$
such that the subgroup generated by $X$ is non-abelian.  Then the
map
\[
r\colon X\times X\to X\times X,\quad r(x,y)=(xyx^{-1},x),
\]
is a non-degenerate solution of the Yang--Baxter equation. This
solution is not involutive.
We claim that $G(X,r)$ is not a Garside group.  Let $G(X,r)$ act
on $X$ by conjugation. Then the center of $G(X,r)$ is the kernel of this
action and hence it has finite index in $G(X,r)$. This implies that all
conjugacy classes of $G(X,r)$ are finite.  Thus the derived subgroup of
$G(X,r)$ is a finite group by a theorem of
Schur~\cite[Theorem~7.57]{MR1307623}. In particular, $G(X,r)$ has torsion
elements and hence $G(X,r)$ is not a Garside group.
\end{example}

\begin{rem}
In~\cite{MR2764830} it is conjectured that structure groups of finite
non-degenerate solutions are Garside groups. Example~\ref{exa:rack} shows
that this conjecture is not true.
\end{rem}

\begin{convention}
A \emph{solution} $(X,r)$ of the YBE will always be a
non-degenerate involutive set-theoretic solution of the
Yang--Baxter equation.
\end{convention}

A left \emph{brace} is an abelian group $(A,+)$ with another group
structure with multiplication $A\times A\to A$, $(a,b)\mapsto ab$, such that
\[
a(b+c)+a=ab+ac,\quad
a,b,c\in A.
\]
It is known that in any left brace $A$ the neutral elements of the groups
$(A,+)$ and $(A,\cdot)$ coincide.  If $A$ is a left brace, then the map
$\lambda\colon (A,\cdot)\to\Aut(A,+)$ given by $\lambda_a(b)=ab-a$ is a group
homomorphism. It follows from the definition that $ab=a+\lambda_a(b)$ and
$a+b=a\lambda^{-1}_a(b)$ for all $a,b\in A$.

An \emph{ideal} $I$ of a left brace $A$ is a normal subgroup $I$ of the
multiplicative group of $A$ such that $\lambda_a(y)\in I$ for all $a\in A$ and
$y\in I$.  The \emph{socle} of a left brace $A$ is defined as the set
\[
\Soc(A)=\{a\in A:\lambda_a=\id\}=\{a\in A:a+b=ab\text{ for all $b\in A$}\}.
\]
The socle of $A$ is an ideal of $A$.

Rump proved that each left brace $A$ produces a solution of the
YBE
\[
    r_A\colon A\times A\to A\times A,\quad
    r_A(a,b)=(\lambda_a(b),\lambda^{-1}_{\lambda_a(b)}(a)).
\]

One of the main results of~\cite{MR1722951} is the following: If
$(X,r)$ is a solution of the YBE, then there exists a
bijective $1$-cocycle $G(X,r)\to \Z^{(X)}$, where $\Z^{(X)}$ is the
free abelian group on $X$. From this it immediately follows that the
canonical map $\iota\colon X\to G(X,r)$ is injective. Now using the
language of braces the existence of a bijective $1$-cocycle can be
written as follows: If $(X,r)$ is a solution, there exists a unique
left brace structure over the structure group $G(X,r)$ such that the
additive group of $G(X,r)$ is isomorphic to $\Z^{(X)}$, the
multiplicative structure is that of $G(X,r)$ and such that
\[
r_{G(X,r)}(\iota\times\iota)=(\iota\times\iota)r.
\]

The \emph{permutation group} $\cG(X,r)$ of a solution $(X,r)$ of the
YBE is defined as the group generated by $\{\sigma_x:x\in X\}$,
where $r(x,y)=(\sigma_x(y),\gamma_y(x))$. It is known that the
map $x\mapsto \sigma_x$ extends to a homomorphism of groups
$\pi\colon G(X,r)\to\cG(X,R)$ such that $\Ker(\pi)=\Soc(G(X,r))$ and
therefore $\cG(X,r)$ has a unique structure of left brace such that
the group isomorphism
\[
\overline{\pi}\colon G(X,r)/\Soc(G(X,r))\to \cG(X,r)
\]
induced by $\pi$ is an isomorphism of left braces.

\begin{rem}
    \label{rem:ideals}
    Let $B$ be a left brace. Using the operation
    \[
    a*b=ab-a-b=(\lambda_a-\id)(b),\quad
    a,b\in B,
    \]
    Rump introduced
    the series
    \[
    B=B^{(1)}\supseteq B^{(2)}\supseteq B^{(3)}\supseteq\cdots,
    \]
    where $B^{(m+1)}=B^{(m)}*B$ is the additive group generated by
    \[
    \{(\lambda_a-\id)(b):a\in B^{(m)},\,b\in B\}
    \]
    for all $m\geq1$.  As a corollary of~\cite[Proposition 6]{MR2278047} Rump
    proved that each $B^{(m)}$ is an ideal of $B$.  Notice that
    this corollary refers to right braces.
\end{rem}

For any left brace $A$ and a subset $X\subseteq A$ we will denote by $\langle
X\rangle$ the subgroup of $(A,\cdot)$ generated by $X$. Similarly $\langle
X\rangle_+$ will denote the subgroup of $(A,+)$ generated by $X$.

\begin{rem}
    \label{rem:generators}
    Let $(X,r)$ be a finite solution of the YBE and let $G=G(X,r)$. Let $m$ be
    a positive integer and let $X_1,\dots,X_r$ be the orbits of $X$ under the
    action of $G^{(m)}$. Then
    \begin{align*}
        G^{(m+1)}&=\langle (\lambda_a-\id)(b):a\in G^{(m)},\,b\in G\rangle_+\\
        &=\langle (\lambda_a-\id)(x):a\in G^{(m)},\,x\in X\rangle_+\\
        &=\langle y-x:x,y\in X_i,\,1\leq i\leq r\rangle_+.
    \end{align*}
    The second equality follows from the fact that $(G,+)$ is generated by $X$
    and $\lambda_a$ is an automorphism of $(G,+)$. The third equality is
    obtained using that $\lambda_a(x)\in X$ for all $x\in X$ and all $a\in G$.
\end{rem}

\section{Multipermutation solutions}

Let $(X,r)$ be a solution of the YBE. Consider the equivalence
relation on $X$ given by $x\sim y$ if and only if
$\sigma_x=\sigma_y$. The \emph{retraction} of $(X,r)$ is defined as
the solution $\Ret(X,r)$ induced by this equivalence relation. One
defines recursively $\Ret^{m+1}(X,r)=\Ret(\Ret^m(X,r))$ for all $m$.
A solution $(X,r)$ of the YBE is said to be a
\emph{multipermutation} solution of level $m$ if $m$ is the minimal
positive integer such that $\Ret^{m}(X,r)$ has only one element. A
solution $(X,r)$ of the YBE is said to be \emph{irretractable}
if $\Ret(X,r)=(X,r)$.

Recall that a group $G$ is said to be \emph{poly}-$\Z$ if it has a subnormal
series
\[
\{1\}=G_0\triangleleft G_1\triangleleft\cdots\triangleleft G_n=G
\]
such that each quotient $G_i/G_{i-1}$ is isomorphic to $\Z$. A group
$G$ is said to be \emph{left orderable} if there is a total order
$<$ on $G$ such that for any $x,y,z\in G$, $x<y$ implies
$zx<zy$.

The main result of the paper is the following theorem.

\begin{thm}
Let $(X,r)$ be a finite non-degenerate involutive
set-theoretic solution of the Yang--Baxter equation. Then the
following statements are equivalent:
\begin{enumerate}
\item $(X,r)$ is a multipermutation
solution.
\item $G(X,r)$ is left orderable.
\item $G(X,r)$ is poly-$\Z$.
\end{enumerate}
\end{thm}

\begin{proof}
Let $G=G(X,r)$ and $\cG=\cG(X,r)$.  By~\cite[Theorem
1.6]{MR1637256}, $G$ is a Bieberbach group. Hence the equivalence
between $(2)$ and $(3)$ follows from~\cite[Theorem 23]{MR639438}.
The implication $(1)\implies(3)$ is~\cite[Proposition
4.2]{MR2189580}; see also~\cite{MR3572046} for another  proof of $(1)\Rightarrow(2)$.
Let us prove $(2)\implies(1)$. For that purpose let
us assume that $(X,r)$ is not a multipermutation solution.
By~\cite[Theorem 5.15]{GI15}, the solution $(G,r_G)$ is not a
multipermutation solution. This implies that the solution
$(\cG,r_{\cG})$ is not a multipermutation solution.
Using~\cite[Proposition 6]{MR3574204} one obtains that
$G^{(m)}\ne\{0\}$ and $\cG^{(m)}\ne\{0\}$ for all $m$. Since $\cG$
is finite, there exists $m$ such that $\cG^{(m+1)}=\cG^{(m)}\ne0$.
By~\cite[Theorem 23]{MR639438}, to prove that $G$ is not left
orderable it suffices to prove that the non-trivial subgroup
$H=G^{(m+1)}$ of $(G,\cdot)$ has trivial center. Let $z\in Z(H)$.
Since $\Soc(G)$ has finite index in $G$ and $G$ is torsion free,
without loss of generality we may assume that $z\in\Soc(G)$. Notice
that if $h\in H$, then
\begin{equation}
    \label{eq:lambda}
\lambda_h(z)=hz-h=zh-h=z+h-h=z.
\end{equation}
Let $X_1,\dots,X_r$ be the orbits of $X$ under the action of
$\cG^{(m)}$. These orbits are the orbits of $X$ under the action of
$G^{(m)}$ through the map $\lambda$. Since $(G,+)$ is the free
abelian group with basis $X$, the element $z$ can be uniquely
written as $z=z_1+\cdots+z_r$, where each $z_i\in\langle
X_i\rangle_+$. From the uniqueness of the decomposition of $z$
and~\eqref{eq:lambda} one obtains that $\lambda_h(z_i)=z_i$ for all
$i\in\{1,\dots,r\}$ and $h\in H$. Now write each $z_i$ as
\[
z_i=\sum_{t\in X_i}n_tt,
\]
where each $n_x\in\Z$.
Remark~\ref{rem:generators} implies that $\sum_{t\in X_i}n_t=0$.
This decomposition is unique since $(G,+)$ is the free
abelian group with basis $X$. Let $x,y\in X_i$ be such that $x\ne y$. Then
there exists $g\in G^{(m)}$ such that $\lambda_g(x)=y$. From
$\cG^{(m+1)}=\cG^{(m)}$ it follows
\[
G^{(m)}=G^{(m+1)}+(\Soc(G)\cap G^{(m)})=H+(\Soc(G)\cap G^{(m)}).
\]
Thus $g=g_1+g_2$, where $g_1\in H$ and $g_2\in\Soc(G)\cap G^{(m)}$. Since
$g_2\in\Soc(G)$, $g=g_2g_1$. Therefore
\[
y=\lambda_g(x)=\lambda_{g_2g_1}(x)=\lambda_{g_2}\lambda_{g_1}(x)=\lambda_{g_1}(x).
\]
Since $z_i=\lambda_{g_1}(z_i)=\sum_{t\in X_i}n_t\lambda_{g_1}(t)$,
we conclude that $n_x=n_y$. Since $\sum_{t\in X_i}n_t=0$, it
follows that $n_t=0$ for all $t\in X_i$ and all $i\in\{1,\dots,r\}$.
Therefore $z=0=1$ and the result follows.
\end{proof}


\begin{exa}
    \label{exa:B(16,91)}
    Let $X=\{1,2,3,4,5,6,7,8,9,a,b,c,d,e,f,g\}$ and let
    \begin{align*}
        &\sigma_1=\id,
        &&\sigma_2=(37)(48)(bf)(cg),\\
        &\sigma_3=(25)(3b4f)(7c8g)(9dea),
        &&\sigma_4=(25)(3g4c)(7f8b)(9dea),\\
        &\sigma_5=(38)(47)(bg)(cf),
        &&\sigma_6=(34)(78)(bc)(fg),\\
        &\sigma_7=(25)(3c7b)(4g8f)(9dea),
        &&\sigma_8=(25)(3f7g)(4b8c)(9dea),\\
        &\sigma_9=(38)(47)(9e)(ad),
        &&\sigma_a=(34)(78)(9e)(ad)(b,f)(c,g),\\
        &\sigma_b=(25)(3f4b)(7g8c)(9aed),
        &&\sigma_c=(25)(3c4g)(7b8f)(9aed),\\
        &\sigma_d=(9e)(ad)(bg)(cf),
        &&\sigma_e=(37)(48)(9e)(ad)(bc)(fg),\\
        &\sigma_f=(25)(3g7f)(4c8b)(9aed),
        &&\sigma_g=(25)(3b7c)(4f8g)(9aed).
    \end{align*}
    Then $r(x,y)=(\sigma_x(y),\sigma^{-1}_{\sigma_x(y)}(x))$ is an irretractable
    solution of the YBE. Hence its structure group $G(X,r)$ is not left
    orderable. In this case
    \[
    \cG(X,r)=\{\sigma_x:x\in X\}.
    \]
\end{exa}

\begin{exa}
    \label{exa:B(16,318)}
    Let $X=\{1,2,3,4,5,6,7,8,9,a,b,c,d,e,f,g\}$ and let
    \begin{align*}
        &\sigma_1=\id,
        &&\sigma_2=(9e)(ad)(bg)(cf),\\
        &\sigma_3=(34)(78)(9e)(ad)(bf)(cg),
        &&\sigma_4=(34)(78)(bc)(fg),\\
        &\sigma_5=(9a)(bc)(de)(fg),
        &&\sigma_6=(9d)(ae)(bf)(cg),\\
        &\sigma_7=(3 4)(78)(9d)(ae)(bg)(cf),
        &&\sigma_8=(3 4)(78)(9a)(de),\\
        &\sigma_9=(5 6)(78)(de)(fg),
        &&\sigma_a=(5 6)(78)(9dae)(bfcg),\\
        &\sigma_b=(3 4)(56)(9dae)(bgcf),
        &&\sigma_c=(3 4)(56)(bc)(de),\\
        &\sigma_d=(5 6)(78)(9ead)(bgcf),
        &&\sigma_e=(5 6)(78)(9a)(bc),\\
        &\sigma_f=(3 4)(56)(9a)(fg),
        &&\sigma_g=(3 4)(56)(9ead)(bfcg).
    \end{align*}
    Then $r(x,y)=(\sigma_x(y),\sigma^{-1}_{\sigma_x(y)}(x))$ is an irretractable
    solution of the YBE. Hence its structure group $G(X,r)$ is not left
    orderable. In this case
    \[
    \cG(X,r)=\{\sigma_x:x\in X\}\simeq\Z/2\times\D_8,
    \]
    where $\D_8$ denotes the
    dihedral group of size $8$.
\end{exa}

\begin{rem}
    The solutions of Examples~\ref{exa:B(16,91)} and~\ref{exa:B(16,318)}
    correspond to the associated solutions to the only left braces of size $16$
    with trivial socle. This was checked with~\textsf{GAP} and the list of small
    braces of~\cite{GV}.
\end{rem}

\begin{exa}
    \label{exa:B(24,96)}
    Let $X=\{1,2,3,4,5,6,7,8,9,a,b,c,d,e,f,g,h,i,j,k,l,m,n,o\}$. Let
    \begin{align*}
        &\sigma_1=\id,\\
        &\sigma_2=(4ag)(5bh)(6ci)(7jd)(8ke)(9lf),\\
        &\sigma_3=(4ga)(5hb)(6ic)(7dj)(8ek)(9fl),\\
        &\sigma_4=(23)(4ogc)(5nhb)(6mia)(7j)(8l)(9k)(ef),\\
        &\sigma_5=(23)(4cmi)(5bnh)(6aog)(7d)(8f)(9e)(kl),\\
        &\sigma_6=(23)(46)(89)(ac)(dj)(el)(fk)(go)(hn)(im),\\
        &\sigma_7=(4m)(5n)(6o)(ag)(bh)(ci),\\
        &\sigma_8=(4gm)(5hn)(6io)(7jd)(8ke)(9lf),\\
        &\sigma_9=(4am)(5bn)(6co)(7dj)(8ek)(9fl),\\
        &\sigma_a=(23)(4cgo)(5bhn)(6aim)(7j)(8l)(9k)(ef),\\
        &\sigma_b=(23)(4o)(5n)(6m)(7d)(8f)(9e)(ac)(gi)(kl),\\
        &\sigma_c=(23)(4iao)(5hbn)(6gcm)(89)(dj)(el)(fk),\\
        &\sigma_d=(4a)(5b)(6c)(gm)(hn)(io),\\
        &\sigma_e=(7jd)(8ke)(9lf)(amg)(bnh)(coi),\\
        &\sigma_f=(4mg)(5nh)(6oi)(7dj)(8ek)(9fl),\\
        &\sigma_g=(23)(46)(7j)(8l)(9k)(ao)(bn)(cm)(ef)(gi),\\
        &\sigma_h=(23)(4imc)(5hnb)(6goa)(7d)(8f)(9e)(kl),\\
        &\sigma_i=(2 3)(4oai)(5nbh)(6mcg)(89)(dj)(el)(fk),\\
        &\sigma_j=(4g)(5h)(6i)(am)(bn)(co),\\
        &\sigma_k=(4ma)(5nb)(6oc)(7jd)(8ke)(9lf),\\
        &\sigma_l=(7dj)(8ek)(9fl)(agm)(bhn)(cio),\\
        &\sigma_m=(23)(4i)(5h)(6g)(7j)(8l)(9k)(ac)(ef)(mo),\\
        &\sigma_n=(23)(46)(7d)(8f)(9e)(ai)(bh)(cg)(kl)(mo),\\
        &\sigma_o=(23)(4c)(5b)(6a)(89)(dj)(el)(fk)(gi)(mo).
    \end{align*}
    Then $r(x,y)=(\sigma_x(y),\sigma^{-1}_{\sigma_x(y)}(x))$ is an irretractable
    solution of the YBE. Hence its structure group $G(X,r)$ is not left
    orderable. In this case
    \[
    \cG(X,r)=\{\sigma_x:x\in X\}\simeq\Sym_4.
    \]
    The unique subgroup of index two of $\cG(X,r)$ is an ideal of $\cG(X,r)$.
    Hence $\cG(X,r)$ is not a simple brace.
\end{exa}

\begin{rem}
    There are two braces of size $24$ with trivial socle. One is the simple
    brace found in~\cite{BachillerSimple} and the other one is that of
    Example~\ref{exa:B(24,96)}.
\end{rem}

\bibliographystyle{abbrv}
\bibliography{refs}

\vspace{30pt}
 \noindent \begin{tabular}{llllllll}
 D. Bachiller && F. Ced\'o  \\
 Departament de Matem\`atiques &&  Departament de Matem\`atiques \\
 Universitat Aut\`onoma de Barcelona &&  Universitat Aut\`onoma de Barcelona  \\
08193 Bellaterra (Barcelona), Spain    && 08193 Bellaterra (Barcelona), Spain \\
\verb+dbachiller@mat.uab.cat+&&  \verb+cedo@mat.uab.cat+\\
\end{tabular}

\bigskip
\noindent \begin{tabular}{llllllll}
L. Vendramin\\
IMAS--CONICET and Departamento de Matem\'atica\\
FCEN, Universidad de Buenos Aires\\
Pabell\'on~1, Ciudad Universitaria
C1428EGA, Buenos Aires, Argentina\\
\verb+lvendramin@dm.uba.ar+
\end{tabular}

\end{document}